\documentclass[12pt]{amsart}
\usepackage{amsfonts,amssymb,amscd,amsmath,enumerate,verbatim}
\usepackage[latin1]{inputenc}
\usepackage{amscd}
\usepackage{latexsym}
\usepackage{pstcol,pst-plot,pst-3d}
\usepackage{mathptmx}
\usepackage{multicol}

\psset{unit=0.7cm,linewidth=0.8pt,arrowsize=2.5pt 4}
\newpsstyle{fatline}{linewidth=1.5pt}
\newpsstyle{fyp}{fillstyle=solid,fillcolor=verylight}
\definecolor{verylight}{gray}{0.97}
\definecolor{light}{gray}{0.9}
\definecolor{medium}{gray}{0.85}

\def\NZQ{\mathbb}               % the font for N,Z,Q,R,C

\def\ZZ{{\NZQ Z}}

\def\frk{\mathfrak}               % font for "Fraktur"

\def\Phi{{\frk N}}
\def\opn#1#2{\def#1{\operatorname{#2}}} % to make operators
\opn\chara{char} \opn\length{\ell} \opn\pd{pd} \opn\rk{rk}
\opn\projdim{proj\,dim} \opn\injdim{inj\,dim} \opn\rank{rank}
\opn\depth{depth} \opn\grade{grade} \opn\height{height} \opn\bheight{bigheight}
\opn\embdim{emb\,dim} \opn\codim{codim}

\opn\Tr{Tr} \opn\bigrank{big\,rank}
\opn\superheight{superheight}\opn\lcm{lcm}
\opn\trdeg{tr\,deg}%\emph{
\opn\reg{reg} \opn\lreg{lreg} \opn\ini{in} \opn\lpd{lpd}
\opn\size{size}\opn{\mult}{mult} \opn{\rev}{rev}
\opn\div{div} \opn\Div{Div} \opn\cl{cl} \opn\Cl{Cl}
\opn\Spec{Spec} \opn\Supp{Supp} \opn\supp{supp} \opn\Sing{Sing}
\opn\Ass{Ass} \opn\Min{Min}
\opn\Ann{Ann} \opn\Rad{Rad} \opn\Soc{Soc}
\opn\Syz{Syz} \opn\Im{Im} \opn\Ker{Ker} \opn\Coker{Coker}
\opn\Am{Am} \opn\Hom{Hom} \opn\Tor{Tor} \opn\Ext{Ext}
\opn\End{End} \opn\Aut{Aut} \opn\id{id} \opn\ini{in}

\opn\nat{nat}
\opn\pff{pf}%   \pf exists already
\opn\Pf{Pf} \opn\GL{GL} \opn\SL{SL} \opn\mod{mod} \opn\ord{ord}
\opn\Gin{Gin}
\opn\Hilb{Hilb}\opn\adeg{adeg}\opn\std{std}\opn\ip{infpt}
\opn\Pol{Pol}
\opn\sat{sat}
\opn\Var{Var}
\opn\Gen{Gen}
\opn\indmatch{indmatch}

\opn\aff{aff} \opn\con{conv} \opn\relint{relint} \opn\st{st}
\opn\lk{lk} \opn\cn{cn} \opn\core{core} \opn\vol{vol}
\opn\link{link} \opn\star{star}
\opn\gr{gr}

\def\pot#1#2{#1[\kern-0.28ex[#2]\kern-0.28ex]}

\opn\dirlim{\underrightarrow{\lim}}
\opn\inivlim{\underleftarrow{\lim}}
\def\Implies{\ifmmode\Longrightarrow \else
        \unskip${}\Longrightarrow{}$\ignorespaces\fi}
\def\implies{\ifmmode\Rightarrow \else
        \unskip${}\Rightarrow{}$\ignorespaces\fi}
\def\iff{\ifmmode\Longleftrightarrow \else
        \unskip${}\Longleftrightarrow{}$\ignorespaces\fi}

\let\:=\colon
\newtheorem{Theorem}{Theorem}[section]
\newtheorem{Lemma}[Theorem]{Lemma}
\newtheorem{Corollary}[Theorem]{Corollary}
\newtheorem{Proposition}[Theorem]{Proposition}
\newtheorem{Remark}[Theorem]{Remark}

\newtheorem{Example}[Theorem]{Example}

\let\epsilon\varepsilon
\let\phi=\varphi
\let\kappa=\varkappa
\def\qed{\ifhmode\textqed\fi
      \ifmmode\ifinner\quad\qedsymbol\else\dispqed\fi\fi}
\def\textqed{\unskip\nobreak\penalty50
       \hskip2em\hbox{}\nobreak\hfil\qedsymbol
       \parfillskip=0pt \finalhyphendemerits=0}
\def\dispqed{\rlap{\qquad\qedsymbol}}

\opn\dis{dis}
\def\pnt{{\raise0.5mm\hbox{\large\bf.}}}

\opn\Lex{Lex}

\newcommand{\inD}[1][\relax]{\def\argone{#1}\def\temprelax{\relax}
  \ifx\argone\temprelax\right.\else\,\middle|#1\right.{}\fi}

\newif\ifbinary
\binarytrue
\begin{document}

\title{Binomial edge ideals and rational normal scrolls}

\author{Faryal Chaudhry, Ahmet Dokuyucu, Viviana Ene}
%\thanks{}

\address{Faryal Chaudhry, Abdus Salam School of Mathematical Sciences, GC University,
68-B, New Muslim Town, Lahore 54600, Pakistan} \email{chaudhryfaryal@gmail.com}
\address{Ahmet Dokuyucu, Faculty of Mathematics and Computer Science, Ovidius University
Bd. Mamaia 124, 900527 Constanta\\ and Lumina-The University of South-East Europe 
Sos. Colentina nr. 64b, Bucharest,
Romania} \email{ahmet.dokuyucu@lumina.org}

\address{Viviana Ene, Faculty of Mathematics and Computer Science, Ovidius University, Bd.\ Mamaia 124,
 900527 Constanta, Romania, and
 \newline
 \indent Simion Stoilow Institute of Mathematics of the Romanian Academy, Research group of the project  ID-PCE-2011-1023,
 P.O.Box 1-764, Bucharest 014700, Romania} \email{vivian@univ-ovidius.ro}
\thanks{The  first author was supported by  by the Higher Education Commission of Pakistan and the Abdus Salam School of Mathematical Sciences, Lahore, Pakistan. The third author was supported by the grant UEFISCDI,  PN-II-ID-PCE- 2011-3-1023.}

\begin{abstract}
Let $X=\left(
\begin{array}{llll}
	x_1 & \ldots & x_{n-1}& x_n\\
	x_2& \ldots & x_n & x_{n+1}
\end{array}\right)$ be the Hankel matrix of size $2\times n$  and let $G$ be a closed graph on the vertex set $[n].$ We study the binomial  ideal $I_G\subset K[x_1,\ldots,x_{n+1}]$ which is generated by all the $2$-minors of $X$ which correspond to the edges of $G.$ We show that $I_G$ is Cohen-Macaulay. We find the minimal primes of $I_G$ and show that $I_G$ is a set theoretical complete intersection. Moreover, a sharp upper bound for the regularity of $I_G$ is given.
\end{abstract}
\maketitle

\section*{Introduction}

Let $K$ be a field and $S=K[x_1,\ldots, x_{n+1}]$ the polynomial ring in $n+1$ variables over the field $K.$ The $2$-minors
of the matrix $X=\left(
\begin{array}{llll}
	x_1 & \ldots & x_{n-1}& x_n\\
	x_2& \ldots & x_n & x_{n+1}
\end{array}\right)
$ generate the ideal $I_X$ of the rational normal curve ${\mathcal X}\subset \mathbb{P}^n.$ It is well-known that $S/I_X$ is Cohen-
Macaulay and has an $S$--linear resolution. We refer the reader to \cite{Ei}, \cite{C}, \cite{BV} for  properties of the ideal of the rational normal scroll. 

On the other hand, in the last few years, the so-called  binomial edge ideals have been intensively studied. They are 
defined as follows. Given a simple graph $G$ on the vertex set $[n]$ with edge set $E(G),$ one considers the ideal $J_G$ 
generated by all the minors $f_{ij}=x_iy_j-x_jy_i$ of the matrix $\left(
\begin{array}{llll}
	x_1 & \ldots & x_{n-1}& x_n\\
	y_1& \ldots & y_{n-1} & y_n
\end{array}\right)$ in the polynomial ring $R=K[x_1,\ldots,x_n,y_1,\ldots,y_n].$ Binomial edge ideals were defined in \cite{
HHHKR} and \cite{Oh}.

In analogy to this construction, in this paper we consider the following ideals in $S$. For a simple graph $G$ on the vertex 
set $[n]$, let $I_G$ be the ideal generated by the $2$-minors $g_{ij}=x_ix_{j+1}-x_jx_{i+1}$ of $X$  with $i<j$ and $\{i,j
\}\in E(G).$ We call $I_G$ the {\em binomial edge ideal of $X$. }

It is clear already from the beginning that unlike the case of classical binomial edge ideals, the ideal $I_G$ strongly 
depends on the labeling of the graph $G.$ For example, if $G$ is the graph displayed in Figure~\ref{fig1}, we get 
$\dim(S/I_G)=3$ for the labeling given in Figure~\ref{fig2} (a) and $\dim(S/I_G)=4$ for the labeling of $G$ given in Figure~
\ref{fig2} (b). 

\begin{figure}[hbt]
\begin{center}
\psset{unit=1cm}
\begin{pspicture}(2,1)(4,3)
\psline(1,1)(2,2)
\psline(2,2)(1,3)
\psline(2,2)(4,2)
\psline(5,1)(4,2)
\psline(5,3)(4,2)

\rput(1,1){$\bullet$}
\rput(2,2){$\bullet$}
\rput(1,3){$\bullet$}
\rput(4,2){$\bullet$}
\rput(5,1){$\bullet$}
\rput(5,3){$\bullet$}

\end{pspicture}
\end{center}
\caption{}
\label{fig1}
\end{figure}

\begin{figure}[hbt]
\begin{center}
\psset{unit=1cm}
\begin{pspicture}(0,0)(2,3)
\psline(2,1)(3,2)
\psline(3,2)(2,3)
\psline(3,2)(5,2)
\psline(6,1)(5,2)
\psline(6,3)(5,2)

\rput(2,1){$\bullet$}
\rput(3,2){$\bullet$}
\rput(2,3){$\bullet$}
\rput(5,2){$\bullet$}
\rput(6,1){$\bullet$}
\rput(6,3){$\bullet$}
\rput(1.8,1){$1$}
\rput(1.8,3){$2$}
\rput(2.7,2){$3$}
\rput(5.3,2){$4$}
\rput(6.2,1){$6$}
\rput(6.2,3){$5$}
\rput(4,0){(b)}

\psline(-3,1)(-2,2)
\psline(-2,2)(-3,3)
\psline(-2,2)(0,2)
\psline(1,1)(0,2)
\psline(1,3)(0,2)

\rput(-3,1){$\bullet$}
\rput(-2,2){$\bullet$}
\rput(-3,3){$\bullet$}
\rput(0,2){$\bullet$}
\rput(1,1){$\bullet$}
\rput(1,3){$\bullet$}
\rput(-3.2,1){$1$}
\rput(-3.2,3){$3$}
\rput(-2.3,2){$2$}
\rput(0.3,2){$4$}
\rput(1.2,1){$6$}
\rput(1.2,3){$5$}
\rput(-1,0){(a)}

\end{pspicture}
\end{center}
\caption{}
\label{fig2}
\end{figure}

However, for some classes of graphs $G$ which admit a natural labeling, we may associate with $G$ a unique ideal 
$I_G$ and study its properties. This is the case, for instance, for closed graphs. We recall from \cite{HHHKR} that $G$ is 
closed if it has a labeling with respect to which is closed. A graph $G$ is called closed with respect to its given labeling 
if  for all edges $\{i,j\}$ and $\{i,k\}$ with $j>i<k$ or $j<i>k$, one has $\{j,k\}\in E(G)$. A closed graph $G$ is chordal 
and, therefore, by Dirac's Theorem, its clique complex $\Delta(G)$ is a quasi-forest. We recall that the clique complex $\Delta(G)$ of $G$ is a simplicial complex whose facets are the maximal cliques of $G,$ that is, the maximal complete subgraphs of $G.$ $\Delta(G)$ is a quasi-forest if  the facets $F_1,\ldots
, F_r$ of $\Delta(G)$ have a leaf order, that is, $F_i$ is a leaf of the simplicial complex generated by $F_1,\ldots,F_i$ 
for all $i>1.$ For  a simplicial complex $\Delta$, a facet $F$ is called a leaf if there is another facet $G$
 of $\Delta$ such that for any facet $H\neq F,$ one has $H\cap F\subseteq G\cap F.$ It was shown in 
\cite{EHH} that if $G$ is closed, then we may label the vertices of $G$ such that the facets of $\Delta(G)$, say $F_1,\ldots
, F_r,$ are intervals, $F_i=[a_i,b_i]\subset [n]$ and if we order $F_1,\ldots,F_r$ such that $a_1<a_2<\cdots <a_r,$ then 
this is a leaf order.

The paper is structured as follows. In Section~\ref{GBases}, we show that the generators of $I_G$ form a Gr\"obner basis with respect to the reverse lexicographic order if and only if $G$ is closed with the given labeling. As a consequence of this theorem, we derive that for a closed graph $G,$ the ideal $I_G$ is Cohen-Macaulay of dimension $1+c,$ where $c$ is the number of connected components of $G.$

 In Section~\ref{main}, we study the properties of $I_G$ for a closed graph $G.$ We compute the minimal prime ideals of $I_G$ in Theorem~\ref{minimal}. By using this theorem, we characterize those connected closed graphs $G$ for which $I_G$ is a radical ideal (Proposition~\ref{radical}). In addition, we show in Corollary~\ref{stci}, that $I_G$ is a set-theoretic complete intersection if $G$ is connected and closed. In the last part of Section~\ref{main}, we a give sharp upper bound for the regularity of $I_G$ (Theorem~\ref{reg}) and we show that $I_G$ has a linear resolution if and only if $G$ is a complete graph. 

\section{Gr\"obner bases}
\label{GBases}

Let $G$ be a graph on the vertex set $[n]$ and $I_G\subset S=K[x_1,\ldots, x_n]$ its associated ideal. The main result of this section is the following.

\begin{Theorem}\label{GB}
The generators of $I_G$ form the reduced Gr\"obner basis of $I_G$ with respect to the reverse lexicographic order induced by $x_1>\cdots >x_n>x_{n+1}$ if and only if $G$ is closed with respect to its given labeling.
\end{Theorem} 

\begin{proof}
Let us first assume that the generators form a Gr\"obner basis of $I_G.$ This implies that for any pair of generators 
$g_{ij}=x_ix_{j+1}-x_jx_{i+1}$ and $g_{k\ell}=x_kx_{\ell+1}-x_\ell x_{k+1}$ of $I_G$, the $S$--polynomial $S_{\rev}(g_{ij},
g_{k\ell})$ reduces to zero. Now let $1\leq i<j<k\leq n$ such that $\{i,j\},\{i,k\}\in E(G)$. We have to show that 
$\{j,k\}$ is also an edge of $G.$ We have $$S_{\rev}(g_{ij},g_{ik})=x_ix_{j+1}x_k-x_ix_jx_{k+1}.$$ Since its initial monomial 
is $x_ix_{j+1}x_k$, $g_{jk}$ must be a generator of $I_G,$ thus $\{j,k\}$ is an edge of $G.$ In a similar way we argue if 
$n\geq i>j>k\geq 1.$

For the converse, let us assume that $G$ is closed. We show that the $S$--polynomial 
$S_{\rev}(g_{ij},g_{k\ell})$ reduces to zero with respect to the generators of $I_G$ for any two generators $g_{ij},g_{k\ell
}$ of $I_G.$ Note that 
$\ini_{\rev}(g_{ij})=x_jx_{i+1}$ and $\ini_{\rev}(g_{k \ell})=x_{\ell}x_{ k+1}.$ If these two monomials have disjoint 
supports we know that $S_{\rev}(g_{ij},g_{k\ell})$ reduces to zero with respect to $g_{ij},g_{k\ell}$. Assuming that, for 
instance, 
$i<k,$ we have to consider the following remaining cases.

Case 1. $\ell=j.$ Then one may check that $S_{\rev}(g_{ij},g_{k\ell})=x_{j+1}g_{ik}$ which is obviously a standard representation.

Case 2. $j=k+1.$ If $\ell=k+1$ we get $S_{\rev}(g_{ij},g_{k\ell})=x_{k+2}g_{ik}$. If $\ell>k+1,$ we obtain 
$S_{\rev}(g_{ij},g_{k\ell})=x_ig_{k+1, \ell}-x_{\ell+1}g_{ik}$  which is again a standard representation.

Therefore, in all cases, the $S$-polynomials $S_{\rev}(g_{ij},g_{k\ell})$ reduce to zero with respect to the generators of 
$I_G.$
\end{proof}

As in the case of classical binomial edge ideals associated with graphs, the ideal $I_G$ where $G$ is the line graph on $n$ vertices has nice properties. 

Let $G$ be a line graph on $[n]$ with $E(G)=\{\{i,i+1\}: 1\leq i\leq n-1\}.$ Then $I_G$ is minimally generated by 
$\{g_{i,i+1}=x_{i+1}^2-x_ix_{i+2}:1\leq i\leq n\}$ and $\ini_{\rev}(I_G)=(x_2^2,x_3^2,\ldots,x_n^2)$. As $x_2^2,x_3^2,\ldots
,x_n^2$ is a regular sequence in $S$, it follows that the generators of $I_G$ form a regular sequence as well. Consequently, 
the Koszul  complex of the generators of $I_G$ gives the minimal free resolution of $S/I_G$ over $S.$ 

The following proposition shows that, for a closed graph $G$, the initial ideal of $I_G$ with respect to the reverse lexicographic order has a simple structure.

\begin{Proposition}\label{initial}
Let $G$ be a closed graph on $[n]$ with $\Delta(G)=\langle F_1,\ldots,F_r\rangle$ where $F_i=[a_i,b_i]$ for $1\leq i\leq r,$ 
and $1=a_1<\cdots <a_r<b_r=n.$ Then $\ini_{\rev}(I_G)$ is a primary monomial ideal, hence it is Cohen-Macaulay.
\end{Proposition}

\begin{proof}
We only need to observe that $I_F,$ where $F=[a,b]$ is a clique, has the initial ideal $\ini_{\rev}(I_F)=(x_{a+1},\ldots, x_b
)^2.$ Then, as $\ini_{\rev}(I_G)=\ini_{\rev}(I_{F_1})+\cdots +\ini_{\rev}(I_{F_r}),$ the conclusion follows.
\end{proof}

%By applying \cite[Corollary 3.3.5]{HH10}, we get the following consequence of the above proposition.

\begin{Corollary}\label{corCM}
Let $G$ be a closed graph. Then $I_G$ is a Cohen-Macaulay ideal of $\dim(S/I_G)=1+c$ where $c$ is the number of connected components of $G$.
\end{Corollary}

\begin{proof}
$I_G$ is a Cohen-Macaulay ideal by \cite[Corollary 3.3.5]{HH10} and $$\dim(S/I_G)=\dim(S/\ini_{\rev}(I_G))=1+c,$$ the last equality being obvious by the form of $\ini_{\rev}(I_G).$ 
\end{proof}

\section{Properties of the  scroll binomial edge ideals of closed graphs}
\label{main}

In this section we study several algebraic and homological  properties of the ideal $I_G$ where $G$ is a closed graph on the 
vertex set $[n].$

\subsection{Associated primes} We recall that $I_X$ denotes the binomial edge ideal associated with the complete graph $K_n.$ It is well known that $I_X$ is a prime ideal.

\begin{Proposition}\label{IX}
Let $G$ be an arbitrary connected graph on the vertex set $[n].$ Then $I_X$ is a minimal prime of $I_G$. If $P$ is a minimal 
prime ideal of $I_G$ which contains no variable, then $P=I_X.$
\end{Proposition}

\begin{proof}
Let $x=\prod_{i=1}^{n+1} x_i.$ We claim that $I_X$ is equal to the saturation of $I_G$ with respect to $x,$ that is, 
$I_X=I_G:x^{\infty}.$ This will be enough to prove the statement of our proposition. Indeed, if $P$ is a minimal prime ideal 
of $I_G$ which does not contain any variable, then $P\supset I_G:x^{\infty}=I_X\supset I_G.$ Since $I_X$ is a prime ideal, it 
follows that $P=I_X.$ 

To prove our claim we first observe that $I_G\subset I_X$ implies that $I_G:x^{\infty}\subset I_X:x^{
\infty}=I_X.$ For the other inclusion, we show that any minimal generator $\delta_{ij}=x_ix_{j+1}-x_jx_{i+1}$ belongs to 
$I_G:x^{\infty}.$ Let $1\le i<j\leq n.$ Since $G$ is connected, there exists a path in $G$ from $i$ to $j$. We prove that 
$\delta_{ij}\in I_G:x^{\infty}$ by induction on the length $r$ of the path. If $\{i,j\}\in E(G)$, there is nothing to prove. 
Let $r>1$ and let $i,i_1,\ldots, i_{r-1},i_r=j$ be a path from $i$ to $j.$ By induction, $\delta_{i,i_{r-1} }\in I_G:x^{
\infty}$. We also have $\delta_{i_{r-1}j}\in I_G:x^{\infty}.$ Then 
$x_{i_{r-1}+1}\delta_{ij}=x_{j+1}\delta_{i i_{r-1}}+x_{i+1}\delta_{i_{r-1}j}\in I_G:x^{\infty},$ therefore, 
$\delta_{ij}\in I_G:x^{\infty}$.
\end{proof}

Now we restrict our study to ideals associated with connected closed graphs.

\begin{Theorem}\label{minimal}
Let $G\neq K_n$ be a connected closed graph on the vertex set $[n]$   and $I_G$ its associated ideal. Then
\[
\Ass(S/I_G)=\Min(I_G)=\{I_X,(x_2,\ldots,x_n)\}.
\]
\end{Theorem}

\begin{proof}
By Corollary~\ref{corCM} and  Proposition~\ref{IX}, we only need to show that if 
$P$ is a minimal prime of $I_G$ which contains at least one variable, then $P=(x_2,\ldots,x_n).$ Let $P\in \Min(I_G)$ 
such that $x_i\in P$ for some $2\leq i\leq n.$ Let $i<n.$ Then, as $\{i,i+1\}\in E(G),$ we get $x_{i+1}\in P.$ Thus, 
$(x_i,\ldots,x_{n})\subset P.$ If $i=2,$ we get $P\supset (x_2,\ldots,x_n)\supset I_G,$ thus we have $P=(x_2,\ldots,x_n).$ Let now $i>2.$ Since $\{i-2,i-1\}\in E(G),$ we get $x_{i-1}\in P.$ Thus, for $i>2$, we get as well $P=(x_2,\ldots,x_n)$.

Let us now assume that $P\in \Min(I_G)$ and $x_1\in P.$ Since $\{i,i+1\}\in E(G)$ for all $i,$ we get $(x_1,\ldots,x_{n})
\subset P,$ which is impossible since $P$ is minimal. A similar argument shows that $P$ cannot contain $x_{n+1}.$
\end{proof}

As a consequence of the above theorem, we may characterize the radical ideals $I_G$.

\begin{Proposition}\label{radical}
Let $G$ be a connected closed graph on the vertex set $[n].$ Then $I_G$ is a radical ideal if and only if $G=K_n$ or 
$\Delta(G)=\langle [1,n-1],[2,n]\rangle$. 
\end{Proposition}

\begin{proof}
The claim is evident if $G=K_n.$ Let now $G\neq K_n.$ Then, by the above theorem, we have 
$\sqrt{I_G}=I_X\cap(x_2,\ldots,x_n).$ We claim that  $I_X\cap(x_2,\ldots,x_n)=I_H$ where $H$ is the closed graph on $[n]$
whose clique complex is generated by the intervals $[1,n-1]$ and $[2,n].$ We obviously have $I_H\subset I_X\cap(x_2,\ldots,x_
n).$ Let $f\in I_X\cap(x_2,\ldots,x_n).$ Then $f=\sum_{1\leq i<j\leq n}h_{ij}\delta_{ij}$ where $\delta_{ij}$ are the 
generators of $I_X$ and $h_{ij}$ are polynomials in $S.$ We have to show that $h_{1n}\delta_{1n}\in I_H$ because $\delta_{ij}\in I_H$ for all $i<j$ with $(i,j)\neq (1,n).$ Since $\delta_{ij}\in (x_2,\ldots x_n)$ for all $i<j$ such that 
$(i,j)\neq (1,n),$ it follows that $h_{1n}\delta_{1n}\in (x_2,\ldots,x_n)$ which implies that $x_1x_{n+1}h_{1n}\in 
(x_2,\ldots,x_n).$ But $x_1x_{n+1}$ is regular on $S/(x_2,\ldots,x_n).$ Thus $h_{1n}\in (x_2,\ldots,x_n)$. We show that 
for all $2\leq j\leq n,$ we have $x_j\delta_{1n}\in I_H$ which will end our proof. For $j=2$ we have $x_j\delta_{1n}=
x_2(x_1x_{n+1}-x_2x_n)=x_1\delta_{2n}+x_n\delta_{12}\in I_H.$ For $j\geq 3,$ we obtain 
$x_j\delta_{1n}=x_{n+1}\delta_{1,j-1}+x_2\delta_{j-1,n}\in I_H.$
\end{proof}

Theorem~\ref{minimal} has the following nice consequence.

\begin{Corollary}\label{stci}
Let $G$ be a connected closed graph. Then $I_G$ is a set-theoretic complete intersection.
\end{Corollary}

\begin{proof}
The statement is known for $G=K_n$ \cite{BV}. Let now $G\neq K_n$ and let $P_n$ be the line graph on $n$ vertices. Obviously, the generators of $I_{P_n}$ are generators for $I_{G}$ as well. By Theorem~\ref{minimal}, we have 
$\sqrt{I_G}=\sqrt{I_{P_n}}.$ The ideal $I_{P_n}$ is generated by $n-1=\height(I_G)$ polynomials. Therefore, $I_G$ is a set-theoretic complete intersection.
\end{proof}

\subsection{Regularity} Let $G$ be a closed graph on the vertex set $[n]$ and $I_G\subset S$ its associated ideal. The first question we may ask is under which conditions on the graph $G$ the ideal $I_G$ has a linear resolution. The next proposition answers this question. We first need the following known statement.

\begin{Lemma}\cite[Exercise 4.1.17 (c)]{BH}\label{exerc}
Let $R=K[x_1,\ldots,x_n]/I$ be a homogeneous Cohen-Macaulay ring. The ring $R$ has an $m$-linear resolution if and only if 
$I_j=0$ for $j<m$ and $\dim_K I_m={m+g-1\choose m}$ where $g=\height I.$
\end{Lemma}

\begin{Proposition}\label{linres}
Let $G$ be a closed graph on $[n].$ Then the following are equivalent:
\begin{itemize}
\item [(a)] $G$ is a complete graph;
\item [(b)] $I_G$  has a linear resolution;
\item [(c)] All powers of $I_G$ have a linear resolution.
\end{itemize}
\end{Proposition}

\begin{proof}
(a)$\Rightarrow$(b) is well known. Let us prove (b)$\Rightarrow$(a). Let $G$ be closed with $c$ connected components, say 
$G_1,\ldots,G_c.$ Since $I_G$ has a $2$--linear resolution, by Lemma~\ref{exerc} and Corollary~\ref{corCM}, it follows that 
$\dim_K(I_G)_2={n-c+1\choose 2}.$ Hence, we get 
\[
{n-c+1\choose 2}=\sum_{i=1}^c \dim_K(I_{G_i})_2\leq \sum_{i=1}^c{n_i\choose 2}
\] where $n_i=|V(G_i)|$ for $1\leq i\leq c.$ The above inequality is equivalent to 
\[
(n-c)(n-c+1)\leq \sum_{i=1}^c n_i(n_i-1).
\] Set $m_i=n_i-1$ for $1\leq i\leq c.$ Then we get the equivalent inequality
\[
(\sum_{i=1}^c m_i)(\sum_{i=1}^c m_i +1)\leq \sum_{i=1}^c m_i(m_i+1)
\] or 
\[
(\sum_{i=1}^c m_i)^2\leq \sum_{i=1}^c m_i^2.
\] This inequality holds if and only if $c=1$, thus $G$ must be connected. Moreover, in this case, since $I_G$ has a linear resolution, we must have $\dim_K(I_G)_2={n\choose 2}=\dim_K(I_{K_n})_2,$ hence $G=K_n.$

 The implication $(c)\Rightarrow(b)$ is trivial, and $(a)\Rightarrow(c)$ is known; see, for example, \cite[Theorem 1]{C} and \cite[Corollary 3.9]{BCV}.
\end{proof}

In the next theorem we give an upper bound for the regularity of $I_G$ when $G$ is a closed graph. 

\begin{Theorem}\label{reg}
Let $G$ be a closed graph on the vertex set $[n].$ Then $\reg(S/I_G)\leq r$ where $r$ is the number of maximal cliques of $G.$
\end{Theorem}

\begin{proof}
Let $H_{S/I_G}(t)$ be the Hilbert series of $S/I_G.$ Then, since $\dim(S/I_G)=1+c,$ where $c$ is the number of connected components of $G,$ we have
\[
H_{S/I_G}(t)=\frac{P(t)}{(1-t)^{1+c}}
\] where $P(t)\in \ZZ[t]$ with $P(1)\neq 0.$ Since $I_G$ is Cohen-Macaulay, we have $\reg(S/I_G)=\deg(P).$

On the other hand, we have
\[
H_{S/I_G}(t)=H_{S/\ini_{rev}(I_G)}(t).
\]

Let us first suppose that $G$ is connected and 
let $F_1,\ldots F_r$ be the maximal cliques of $G$ where $F_i=[a_i,b_i]$ for $1\leq i\leq r$ with 
$1=a_1<a_2<\cdots <a_r<b_r=n.$ Then
{\small
\[
\ini_{\rev}(I_G)=\ini_{\rev}(I_{F_1})+\cdots +\ini_{\rev}(I_{F_r})=(x_2,\ldots, x_{b_1})^2+(x_{a_2+1},\ldots,x_{b_2})^2+\cdots+ 
(x_{a_{r-1}+1},\ldots,x_{n})^2.
\]
}
Then, as $x_1$ and $x_{n+1}$ are regular on $S/\ini_{\rev}(I_G)$, we get
\[
P(t)=H_{S/(\ini_{\rev}(I_G),x_1,x_{n+1})}(t)=h_0+h_1t+\cdots +h_qt^q
\] where $q=\deg(P)$ and $h_i=\dim(S/(\ini_{\rev}(I_G),x_1,x_{n+1}))_i$ for $0\leq i\leq q.$

In order to prove our statement, it is enough to show that $q\leq r.$ Let $i>r.$ We have to show that 
$\dim(S/(\ini_{\rev}(I_G),x_1,x_{n+1}))_i=0.$ But $\dim(S/(\ini_{\rev}(I_G),x_1,x_{n+1}))_i$ is equal to the number of squarefree monomials $w=x_F$ in the variables $x_2,\ldots,x_n$ such that $x_F\notin \ini_<(I_G)$ and $\deg x_F=i.$ Let $F=\{j_1,\ldots,j_i\}$ with $2\leq j_1<\cdots <j_i\leq n.$ Since $\deg x_F\geq r+1$, there exists $1\leq p<q\leq i$ such that 
$j_p$ and $j_q$ belong to the same clique $F_\ell$ of $G.$ This implies that $x_F\in\ini_<(I_G).$ Therefore, 
$\dim(S/(\ini_<(I_G),x_1,x_{n+1}))_i=0$ and, consequently, $\reg(S/I_G)=\deg(P)\leq r.$

Now, let $G_1,\ldots,G_c$ be the connected components of $G$ and let $r_i$ the number of cliques of $G_i$ for $1\leq i\leq c.$ We may assume that $V(G_i)=[n_{i}+1, n_{i+1}]$ for some integers 
$0=n_1<\cdots <n_c<n_{c+1}=n.$ We set $S_i=K[\{x_j: n_i+1\leq j\leq n_{i+1}\}]$ for $1\leq i\leq c.$ Let $M_i$ be the set of minimal monomial generators of $\ini_{\rev}(J_{G_i})$ for all $i. $ One observes that any two monomials $u \in M_i,$ $v\in M_j$ with $i\neq j,$ have disjoint supports. This implies that
\[
S/\ini_{\rev}(J_G)\cong\bigotimes_{i=1}^c S_i/\ini_{\rev}(J_{G_i}).
\] Consequently, \[\reg(S/J_G)=\reg(S/\ini_{\rev}(J_{G}))= \sum_{i=1}^c \reg(S_i/\ini_{\rev}(J_{G_i}))
\leq \sum_{i=1}^c r_i=r \]
\end{proof}

\begin{Remark}\label{regr}{\em
The upper bound given in the above theorem is sharp. Indeed, let $G$ be a closed graph with the maximal cliques $F_i=[a_i,a_{i+1}]$ where $1=a_1<a_2<\cdots<a_r<a_{r+1}=n.$ In this case, we have
\[
\ini_{\rev}(I_G)=(x_2,\ldots,x_{a_2})^2+(x_{a_2+1},\ldots,x_{a_3})^2+\cdots+(x_{a_r+1},\ldots,x_n)^2.
\]
Therefore,
\[
S/(\ini_{\rev}(I_G),x_1,x_{n+1})\cong (S_1/(x_2,\ldots,x_{a_2})^2)\otimes_K\cdots \otimes_K(S_r/(x_{a_r+1},\ldots,x_{n})^2)
\]
where $S_i=K[x_{a_i+1},\ldots,x_{a_{i+1}}]$ for all $i,$
which implies that 
\[
H_{S/(\ini_{\rev}(I_G),x_1,x_{n+1})}(t)=\prod_{i=1}^r(1+(a_{i+1}-a_i)t).
\] This shows that $\reg(S/I_G)=r.$
}
\end{Remark}

From Proposition~\ref{linres} and Theorem~\ref{reg}, we derive the following consequence.

\begin{Corollary}
Let $G$ be a closed graph with two maximal cliques. Then $\reg(S/I_G)=2.$
\end{Corollary}

The following example shows that the inequality given in Theorem~\ref{reg} may be also strict.

\begin{Example}{\em 
Let $G$ be the closed graph on the vertex set $[6]$ with the maximal cliques $F_1=[1,4],\ F_2=[3,5],$ and $F_3=[4,6]$. We have 
$\reg(S/I_G)=2<3.$
}
\end{Example}

{}

\end{document}